\documentclass[12pt]{amsart}
\pdfoutput=1
\usepackage{heuristica}
\usepackage[heuristica,vvarbb,bigdelims]{newtxmath}

\usepackage{fullpage,url,mathrsfs,amssymb,color,multicol,enumerate,amscd}
\usepackage{multirow}
\usepackage{mathtools}
\usepackage[all]{xy} 
\usepackage{verbatim,moreverb}

\usepackage[OT2,T1]{fontenc}

\usepackage{tikz-cd}
\usepackage{xfrac}
\usepackage{float}

\definecolor{backgroundcolor}{rgb}{0.9,0.8,0.8}

\numberwithin{equation}{section}
\usepackage{setspace}
\usepackage{color}
\newcommand{\vanish}[1]{}

\def\bbar#1{\setbox0=\hbox{$#1$}\dimen0=.2\ht0 \kern\dimen0 }
\newcommand{\defi}[1]{\textsf{#1}} 
\usepackage{enumitem}

\newenvironment{romanenum}{\hfill \begin{enumerate}[label=({\roman*})]
}
{\end{enumerate}}
\DeclareSymbolFont{cyrletters}{OT2}{wncyr}{m}{n}
\DeclareMathSymbol{\Sha}{\mathalpha}{cyrletters}{"58}

\newcommand{\PP}{{\mathbb P}}
\newcommand{\Q}{{\mathbb Q}}

\newcommand{\ZZ}{{\mathbb Z}}
\newcommand{\Zhat}{{\hat{\ZZ}}}

\def\bbar#1{\setbox0=\hbox{$#1$}\dimen0=.2\ht0 \kern\dimen0 \overline{\kern-\dimen0 #1}}
 
\newcommand{\Qab}{{\mathbb Q}^{\operatorname{ab}}}

\newcommand{\frakA}{{\mathfrak A}}
\newcommand{\frakB}{{\mathfrak B}}

\newcommand{\calF}{{\mathcal F}}

\newcommand{\calH}{{\mathcal H}}

\DeclareMathOperator{\Aut}{Aut}
\DeclareMathOperator{\Gal}{Gal}

\newcommand{\GL}{\operatorname{GL}}
\newcommand{\SL}{\operatorname{SL}}
\newcommand{\PGL}{\operatorname{PGL}}

\newcommand{\injects}{\hookrightarrow}

\newcommand{\intersect}{\cap} 
\newcommand{\tensor}{\otimes}

\def\CC{\mathbb C}

\usepackage{amsthm}

\definecolor{webcolor}{rgb}{0,0,1}
\definecolor{webbrown}{rgb}{.6,0,0}
\usepackage[
        colorlinks,
        linkcolor=webbrown,  filecolor=webbrown,  citecolor=webbrown, 
        backref,
        pdfauthor={Rakvi}, 
        pdftitle={A Classification of Genus 0 Modular Curves with a Rational Point},
]{hyperref}
\usepackage{cleveref}
\showboxdepth=\maxdimen
\showboxbreadth=\maxdimen
\newtheorem{theorem}{Theorem}[section]
\newtheorem{lemma}[theorem]{Lemma}

\newtheorem{proposition}[theorem]{Proposition}

\theoremstyle{definition}
\newtheorem{definition}[theorem]{Definition}

\newtheorem*{Mazur's Program B}{Mazur's Program B}

\theoremstyle{remark}
\newtheorem*{remark}{Remark}

\crefname{section}{§}{§§}
\crefname{lemma}{Lemma}{Lemmas}
\crefname{equation}{equation}{equations}
\crefname{theorem}{Theorem}{Theorems}
\crefname{proposition}{Proposition}{Propositions}

\usepackage[alphabetic,backrefs,lite]{amsrefs} 
\usepackage{datetime}
\date{\today}
\usepackage{mathtools}
\usepackage{amsfonts}
\usepackage{hyperref}
\usepackage{microtype}
\begin{document}

\title[]{Genus $0$ Modular curves of prime power level with a point defined over number fields other than $\Q$}

\author{Rakvi}

\begin{abstract}

Associated to an open subgroup $G$ of $\GL_2(\Zhat)$ satisfying conditions $-I \in G$ and $\det(G) \subsetneq (\Zhat)^{\times}$ there is a modular curve $X_G$ which is a smooth compact curve defined over an extension of $\Q.$ In this article, we give a complete list of all such prime power level genus $0$ modular curves with a point.

\end{abstract}

\maketitle
\setcounter{tocdepth}{1}

\section{Notations} 

We denote the $N$-th cyclotomic field by $K_N$. We denote the number of divisors of a positive integer $n$ by $\omega(n).$ For a field $K$, we denote the multiplicative group of non-zero elements of $K$ by $K^{\times}.$ We denote the profinite completion of ring of integers $\ZZ$ by $\Zhat$ and the multiplicative group of its units by $\Zhat^{\times}.$

\section{Introduction}

Let $G$ be an open subgroup of $\GL_2(\Zhat)$ satisfying $-I \in G$ and $\det(G) \subsetneq \Zhat^{\times}.$ Associated to such a subgroup we have a modular curve $X_G$ which is a smooth, projective curve defined over a number field $K \ne \Q$ along with a morphism to $\PP^1_K.$ In \cref{sec:Modular Curves}, we discuss this in detail. 

We say that $G$ has level $M$ if it is the smallest positive integer such that $G$ is equal to inverse image of $\pi_M(G)$, where $\pi_M \colon \GL_2(\Zhat) \to \GL_2(\ZZ/M\ZZ)$ is the reduction modulo $M.$ In this article we give a list of all such genus $0$ modular curves $(X_G,\pi_G)$ of prime power level with a point such that $\det(G) \subsetneq \Zhat^{\times}$ and $-I \in G.$ We performed all the computations on \texttt{Magma.} For a complete list of genus $0$ and genus $1$ modular curves $(X_G,\pi_G)$ defined over $\Q$ that have infinitely many rational points, see \cite{MR3671434}.

\begin{theorem}\label{Mainthm}
Let $G$ be an open subgroup of $\GL_2(\Zhat)$ satisfying the following conditions:
\begin{itemize}
    \item It contains $-I$ and $\det(G) \subsetneq \Zhat^{\times}.$
    
    \item The level of $G$ is a power of some prime number $p.$
    
    \item The genus of associated modular curve $(X_G,\pi_G)$ is $0$ and has a point.
\end{itemize}
Then, $(X_G,\pi_G)$ is defined over a subfield $K \neq \Q$ of $K_m$ where $m \in \{3,5,7,9,8,11,13,16,25,27,32\}.$ We list all such curves that have odd prime power level \href{https://github.com/Rakvi6893/Genus-0-modular-curves-over-higher-number-fields/blob/main/Tables%20for%20odd%20prime%20power%20level_genus%200%20over%20extensions%20of%20Q.txt}{in text file 1}

and curves whose level is a power of $2$ \href{https://github.com/Rakvi6893/Genus-0-modular-curves-over-higher-number-fields/blob/main/Table_powers%20of%202.txt}{in text file 2}.

\end{theorem}

\section{Modular Curves and Modular Functions} \label{sec:Modular Curves}

\subsection{Congruence subgroups}

Let $\Gamma$ be a congruence subgroup of $\SL_{2}(\ZZ).$ \label{SL2level}The level of $\Gamma$ is the smallest positive integer $N$ for which $\Gamma$ contains \[\Gamma(N):=\{A \in \SL_{2}(\ZZ)~|~A \equiv I\pmod{N}\}.\] The group $\SL_2(\ZZ)$ acts on the complex upper half plane $\calH$ by linear fractional transformations. 

We define the extended upper half plane $\calH^{*}$ as $\calH \cup \Q \cup \{\infty\}.$ The action of $\SL_2(\ZZ)$ on $\calH$ naturally extends to the action on $\calH^{*}.$ A \defi{cusp} of $\Gamma$ is a $\Gamma$-orbit of $\Q \cup \{\infty\}.$ After adding cusps to the quotient $\Gamma \setminus \calH$, we get a smooth compact Riemann surface which we denote by $X(\Gamma).$ We will use the notation $X_N$ to denote $X(\Gamma(N))$, where $N$ is a positive integer.

\subsection{Modular functions}

Let $\Gamma$ be a congruence subgroup of $\SL_2(\ZZ).$ A \defi{modular function} $f$ for $\Gamma$ is a meromorphic function of $X(\Gamma)$ i.e., $f$ is a meromorphic function on the upper half plane $\calH$ that satisfies $f(\gamma \tau)=f(\tau)$ for all $\gamma \in \Gamma$ and is meromorphic at the cusps.   

The \defi{width} of the cusp at $\infty$ is the smallest positive integer $w$ such that $(\begin{smallmatrix}
1 & w \\ 
0  & 1
\end{smallmatrix}) \in \Gamma.$ Define $q:=e^{2\pi \dot{\iota} \tau}$ and $q^{1/w}:=e^{2\pi \dot{\iota} \tau/w}$, where $\tau \in \calH.$ If $f$ is a modular function for $\Gamma$, then $f$ has a unique $q$-expansion $f(\tau)=\sum \limits_{n \in \ZZ} c_{n}q^{n/w}$ with $c_{n} \in \CC$ and $c_n=0$ for all but finitely many $n$ less than $0.$

Fix a positive integer $N.$ Let $\calF_N$ denote the field of meromorphic functions of the Riemann surface $X_N$ whose \defi{$q$-expansion} has coefficients in $K_{N}$. For $N=1$, we have $\calF_1=\Q(j)$ where $j$ is the modular $j$-invariant. The first few terms of $q$-expansion of $j$ are $q^{-1}+744+196884q+21493760q^2+\dotsb.$
If $N'$ is a divisor of $N$, then $\calF_{N'} \subseteq \calF_N.$ In particular, we have that $\calF_1 \subseteq \calF_N.$

Take $g \in \GL_2(\ZZ/N\ZZ)$ and let $d=\det(g).$ Then, $g=(\begin{smallmatrix}
1 & 0 \\
0 & d \end{smallmatrix})g'$ for some $g' \in \SL_2(\ZZ/N\ZZ)$. There is a unique right action $*$ of $\GL_2(\ZZ/N\ZZ)$ on $\calF_N$ that satisfies the following properties (see Chapter $2$, section $2$ in \cite{MR648603} for details):

\begin{itemize}
\item If $g \in \SL_2(\ZZ/N\ZZ)$, then \[f*g:=f|_{\gamma},\] where $\gamma \in \SL_{2}(\ZZ)$ is congruent to $g$ modulo $N$ and $f|_{\gamma}(\tau)=f(\gamma \tau)$.

\item If $g=(\begin{smallmatrix}
1 & 0 \\
0 & d \end{smallmatrix})$, then 
\[f*g:=\sigma_{d}(f):=\sum\limits_{n \in \ZZ} \sigma_{d}(c_{n})q^{n/w},\] where $f=\sum\limits_{n \in \ZZ} c_{n}q^{n/w}$ and $\sigma_{d}$ is the automorphism of $K_N$ that satisfies $\sigma_{d}(\zeta_N)=\zeta_N^{d}.$
\end{itemize} 

\begin{proposition}The following properties hold.\label{action on F_N}
\hfill
\begin{romanenum}
    \item $-I$ acts trivially on $\calF_N$ and the action $*$ of $\GL_{2}(\ZZ/N\ZZ)/\{\pm I\}$ on $\mathcal{F}_N$  is faithful.
    
    \item We have $\calF_N^{\GL_2(\ZZ/N\ZZ)/\{\pm I\}}=\Q(j).$
    
    \item \label{opposite}The extension $\mathcal{F}_N$ over $\Q(j)$ is Galois and the action $*$ gives an isomorphism \[\GL_{2}(\ZZ/N \ZZ)/\{\pm I\} \to \Gal(\mathcal{F}_N/\Q(j))^{op},\] where $\Gal(\mathcal{F}_N/\Q(j))^{op}$ is the opposite group of $\Gal(\mathcal{F}_N/\Q(j))$ (i.e., the same underlying set with the group law reversed.)
    
    \item The algebraic closure of $\Q$ in $\calF_{N}$ is $K_{N}$. We have $\mathcal{F}_N^{\SL_{2}(\ZZ/N \ZZ)/\{\pm I\}}=K_N(j).$\label{algclosure}

\end{romanenum}

\end{proposition}

\begin{proof}

Refer to Theorem $6.6$, Chapter $6$ in \cite{MR1291394}.
\end{proof}

\begin{remark}

We require the opposite in \cref{action on F_N}, (\ref{opposite}) since $*$ is a right action.
\end{remark}

\subsection{Modular Curves}

Let $G$ be an open subgroup of $\GL_{2}(\Zhat)$ satisfying $-I \in G$ and $\det(G) \subsetneq \Zhat^{\times}.$\label{level G} Recall that the level of $G$ is the smallest positive integer $M$ such that $G$ is the inverse image of its image under the reduction modulo $M$ map $\pi_{M} \colon \GL_2(\Zhat) \to GL_2(\ZZ/M\ZZ).$ 

Let $M$ be the level of $G$. There is a natural action of $\ZZ^{\times}$ on $K_M$ induced by the isomorphism $\Gal(\Qab/\Q) \to \ZZ^{\times}.$ Let $K_M^{\det(G)}$ denote the subfield of $K_M$ fixed under the action of $\det(G).$ Let $G_M \subseteq \GL_2(\ZZ/M\ZZ)$ be the image of $G$ under $\pi_M.$ It will follow that $K_M^{\det(G)}$ is algebraically closed in $\calF_M^{G_M}$ by \cref{action on F_N}, (\ref{opposite}). The field $\calF_M^{G_M}$ has transcendence degree $1$ over $K_M^{\det(G)}$ since it is a finite extension of $K_M(j)$. 

\begin{definition}

The \defi{modular curve} $X_G$ is the smooth compact curve defined over $K_M^{\det(G)}$ whose function field is $\calF_M^{G_M}.$ The map $\pi_G \colon X_G \to \PP^1_{K_M^{\det(G)}}$ is the nonconstant morphism corresponding to the inclusion  $K_M^{\det(G)}(j) \subseteq \calF_M^{G_M}.$
\end{definition}

Let $S$ be a scheme over $\ZZ[1/N].$ One could also define $X_{G}$ as the compactification of the coarse moduli scheme $Y_G$ defined over $\ZZ[1/N,\zeta_N]^{\det(G)}$ that parametrises elliptic curves $E$ defined over $S$ with $G$-level structure. (refer to \cite{MR0337993} chapter $4$, section 3 for details).

Let $\Gamma$ be a congruence subgroup of level $N$ that contains $-I.$ Let $\Gamma_N \subseteq \SL_2(\ZZ/N\ZZ)$ be the image of $\Gamma$ under reduction modulo $N.$

\begin{definition}\label{X_Gamma}

The algebraic curve $X_{\Gamma}$ is the nice curve over $K_N$ whose function field is $\calF_N^{\Gamma_N}.$ The map $\pi_\Gamma \colon X_{\Gamma} \to \PP^1_{K_N}$ is the morphism corresponding to the inclusion $K_N(j) \subseteq \calF_N^{\Gamma_N}.$
\end{definition}

Using the inclusion $K_N \subsetneq \CC$, we can identify $X_{\Gamma}(\CC)$ with the Riemann surface $X(\Gamma).$

Using \cref{action on F_N}, (\ref{algclosure}) we know that the field $K_M(X_G)$ which is the function field of the base extension of $X_G$ to $K_M$ is isomorphic to the function field $K_M(X_{\Gamma})$, where $\Gamma$ is conjugate to the congruence subgroup $G \intersect \SL_2(\ZZ).$ Equivalently, there exists an isomorphism $f \colon (X_{\Gamma})_{K_M} \to (X_G)_{K_M}$ that satisfies $\pi_G \circ f = \pi_{\Gamma}.$

\section{Twists}\label{sec:Twists}

\subsection{Background and preliminaries}

For background in Galois cohomology see section 5.1, Chapter 1 of \cite{MR1466966}. We recall the definitions and concepts that we need for reading this article.  

Let $G$ be a topological group and $A$ be a $G$-group i.e., $A$ is a topological group with the discrete topology, with a continuous $G$-action that respects the group law of $A$. We will assume that $G \times G$ has product topology.

\begin{definition}
\begin{itemize}
    \item A \defi{cocycle} $\zeta \colon G \to A$, is a continuous function satisfying the \defi{cocycle property}, i.e., $\zeta(\sigma \tau)=\zeta(\sigma) \cdot \sigma(\zeta(\tau))$ for every $\sigma,\tau \in G.$ We say that two cocycles $\zeta_1$ and $\zeta_2$ are \defi{cohomologous} if there is an $a \in A$ such that $\zeta_1(\sigma)=a \cdot \zeta_2(\sigma) \cdot \sigma(a^{-1})$ for every $\sigma \in G.$  
    
    \item A \defi{2-cocycle} $\mu \colon G \times G \to A$, is a continuous function satisfying the property \[\sigma_1(\mu(\sigma_2,\sigma_3)) \cdot \mu(\sigma_1,\sigma_2\sigma_3)=\mu(\sigma_1\sigma_2,\sigma_3) \cdot \mu(\sigma_1,\sigma_2)\] for every $\sigma_i \in G$ and $i \in \{1,2,3\}.$ We say that a $2$-cocycle $\mu$ is a $2$-coboundary if there exists a function $\alpha \colon G \to A$ such that $\mu(\sigma,\tau)=\alpha(\sigma) \cdot \sigma(\alpha(\tau)) \cdot \alpha(\sigma\tau)^{-1}$ for every $\sigma,\tau \in G.$ 
\end{itemize}

\end{definition}

It is easy to check that being cohomologous is an equivalence relation. We will denote the set of cocycles modulo this relation by $H^{1}(G,A).$ Let $e \colon G \to A$ be the cocycle which sends $g$ to $1$ for every $g \in G$, where $1$ is the identity element of $A.$ We say that a cocycle $\zeta \colon G \to A$ is a \defi{coboundary} if $\zeta$ is cohomologous to $e.$

Fix a number field $K.$ Fix a nice curve $X$ over $K$ with a nonconstant morphism $\pi_X \colon X \to \mathbb{P}^{1}_{K}.$ Let $L$ be a Galois extension of $K.$ 

\begin{definition} \label{L-twist}

A \defi{$L$-twist} of $(X,\pi_X)$ is a pair $(Y,\pi_Y)$ where $Y$ is a curve over $K$ and $\pi_Y$ is a morphism $Y \to \mathbb{P}^{1}_{K}$ such that there exists an isomorphism $f \colon X_L \to Y_L$ which satisfies $\pi_{Y}\circ f = \pi_{X}.$ \label{isomorphic}We say that $(X,\pi_X)$ and $(Y,\pi_Y)$ are \defi{$K$-isomorphic} if $(Y,\pi_Y)$ is a $K$-twist of $(X,\pi_X).$ 
\end{definition}

\label{Aut} Let $\Aut_L(X,\pi_X)$ be the subgroup of $\Aut(X_L)$ consisting of $f$ such that $\pi_{X} \circ f= \pi_{X}.$ Since $\pi_X$ is nonconstant, $\Aut_L(X,\pi_X)$ is a finite $\Gal(L/K)$-group. 

We will now define a map $\theta$ between the set of $L$-twists of $(X,\pi_X)$ up to $K$-isomorphism and $H^{1}(\Gal(L/K),\Aut_{L}(X,\pi_X))$. 

Let $(Y,\pi_Y)$ be a $L$-twist of $(X,\pi_X).$ So, there is an isomorphism $f \colon X_L \to Y_L$ such that $\pi_{Y}\circ f = \pi_{X}.$ The map $\xi \colon \Gal(L/K) \to \Aut_{L}(X,\pi_X)$ defined as $\xi(\sigma)=f^{-1}\circ \sigma(f)$ is a cocycle. Using that $\pi_{Y}$ and $\pi_{X}$ are defined over $K$, we get that $\pi_{X} \circ \xi(\sigma)= \pi_{X}$ for every $\sigma \in \Gal(L/K).$ We define $\theta((Y,\pi_Y))=[\xi] \in H^{1}(\Gal(L/K),\Aut_{L}(X,\pi_X)).$ 

\begin{proposition}
The map $\theta$ defined above is well-defined and is a bijection between the set of $L$-twists of $(X,\pi_X)$ up to $K$-isomorphism and $H^{1}(\Gal(L/K),\Aut_{L}(X,\pi_X)).$ 
\end{proposition}

\begin{proof}
Proof is similar to as given in Proposition $5.3$ of \cite{2105.14623}.

\end{proof}
Let $(Y,\pi_Y)$ be a $L$-twist of $(X,\pi_X)$. Fix an isomorphism $f \colon X_{L} \to Y_{L}$ that satisfies $\pi_Y \circ f=\pi_X.$ The cocycle $\zeta \colon \Gal(L/K) \to \Aut_{L}(X,\pi_X)$, $\sigma \mapsto f^{-1}\circ \sigma(f)$ is a representative of $\theta((Y,\pi_Y)).$

\begin{lemma}\label{Aut:twists}
 We have $\Aut_{L}(Y,\pi_Y)=\{f \circ g \circ f^{-1}~|~ g \in \Aut_{L}(X,\pi_X)\}.$
    
\end{lemma}

\begin{proof}
Proof is similar to as given in Lemma $5.4$ of \cite{2105.14623}.
\end{proof}

Now assume that the curve $X$ is $\PP^1_{K}.$ Let $(Y,\pi_Y)$ be a $L$-twist  of $(X,\pi_X).$ Fix an isomorphism $f \colon X_L \to Y_L$ and let $\zeta \colon \Gal(L/K) \to \Aut_L(X,\pi_X)$ be the cocycle $\zeta(\sigma)=f^{-1}\circ \sigma(f).$

\begin{lemma}\label{lemma:coboundary P^1}

The curve $Y$ is isomorphic to $\PP^1_{K}$ if and only if $\zeta \colon \Gal(L/K) \to \Aut_L(X,\pi_X) \injects \Aut(X_L)$ is a coboundary in $H^{1}(\Gal(L/K),\Aut(X_L)).$
\end{lemma}

\begin{proof}

Proof is similar to as given in Lemma $5.6$ of \cite{2105.14623}.
\end{proof}

In our application, given a nice curve $(X,\pi_X)$ where $X$ is $\PP^1_{K}$ we need to find $L$-twists $(Y,\pi_Y)$ such that $Y$ is $\PP^1_{K}.$ We do this as follows. We consider the cocycles from $\Gal(L/K)$ to $\Aut_L(X,\pi_X).$ There are finitely many if $\Gal(L/K)$ is finite. We check if a cocycle \[\zeta \colon \Gal(L/K) \to \Aut_L(X,\pi_X) \injects \Aut(X_L)=\PGL_2(L) \] is a coboundary in $H^{1}(\Gal(L/K),\Aut(X_L)).$ If $\zeta$ is a coboundary in $H^{1}(\Gal(L/K),\PGL_2(L))$, then there is an $A \in \PGL_2(L)$ such that  $\zeta(\sigma)=A^{-1}\sigma(A)$ for every $\sigma \in \Gal(L/K)$; $A$ is an isomorphism between $X_L \to \PP^1_{L}$ and using $\pi_X \circ \zeta(\sigma)=\pi_X$ for every $\sigma \in \Gal(L/K)$ we observe that $\pi_X \circ A^{-1} \colon \PP^1_{K} \to \PP^{1}_{K}.$ The pair $(\PP^1_{\Q},\pi_X \circ A^{-1})$ is twist of $(X,\pi_X).$

\subsection{Determining if a cocycle is a coboundary}

In this section, we describe a method that we use to determine whether a cocycle \[\zeta \colon \Gal(L/K) \to \PGL_2(L) \] is a coboundary or not. If it is a coboundary $\Gal(L/K)$ is finite and cyclic, then we also explain how to compute a matrix $A \in \PGL_2(L)$ such that $\zeta(\sigma)=A^{-1}\sigma(A)$ for every $\sigma \in \Gal(L/K).$

Let \[\zeta \colon \Gal(L/K) \to \PGL_2(L)\] be a cocycle. The cocycle $\zeta$ determines a twist $C$ of $\PP^1_{K}$ up to isomorphism. In this section, we explain how to explicitly compute $C$ as a conic in $\PP^2_{K}.$ The Hasse principle can then be used to determine if $C(K)$ is empty or not. If $C(K)$ is nonempty, then it is isomorphic to $\PP^1_{K}$ and equivalently $\zeta$ is a coboundary. When $\zeta$ is a coboundary, we will explain how to compute an $A \in \PGL_{2}(L)$ such that $\zeta(\sigma)=A^{-1}\sigma(A)$ for every $\sigma \in \Gal(L/K).$

Let $Q_0$ be the quadratic form $y^2-xz.$ Let $C_{0}$ be the conic defined by $Q_0=0$ in $\mathbb{P}^{2}_{K}$. 

Let $\phi \colon \PGL_2(L) \to \GL_3(L)$ be the map \[[(\begin{smallmatrix}
a & b \\
c & d
\end{smallmatrix})] \mapsto \frac{1}{ad-bc}\bigg(\begin{smallmatrix}
a^{2} & 2ab & b^{2} \\
ac & ad+bc & bd \\
c^{2} & 2cd & d^{2}    
\end{smallmatrix}\bigg).\] The map $\phi$ comes from remark 3.3 in \cite{MR3906177}. It is an injective group homomorphism that respects the $\Gal(L/K)$-action. The image of $\phi$ gives automorphisms of $\PP^2_L$ that stabilizes $(C_0)_L.$ The map $\phi$ induces an isomorphism $\Bar{\phi} : \Aut(\PP^{1}_{L})=\PGL_2(L) \to \Aut((C_0)_{L})$ that respects the $\Gal(L/K)$- action. One can check that $\Bar{\phi}$ is induced by the isomorphism $\PP^1_{K} \to C_0$ that sends $[x:y]$ to $[x^2:xy:y^2].$

Define the map \[\bar{\zeta}:=\phi \circ \zeta \colon \Gal(L/K) \to \GL_3(L).\] Since $\phi$ is a group homomorphism and respects the $\Gal(L/K)$-action, $\Bar{\zeta}: \Gal(L/K) \to \GL_3(L)$ is a cocycle. By Hilbert 90, there is an $M \in \GL_{3}(L)$ such that $\Bar{\zeta}(\sigma)=M^{-1}\sigma(M)$ for every $\sigma \in \Gal(L/K).$ We explain how to compute $M$ in \cref{Computing H90 mat}.

\begin{lemma}\label{twist as a conic}

The quadratic form $Q:=Q_{0}(M^{-1}(x,y,z)^{T})$ has coefficients in $K$ and the conic $C \subseteq \PP^2_{K}$ defined by $Q=0$ is isomorphic to the twist of $\PP^1_{K}$ by $\zeta.$
\end{lemma}

\begin{proof}
Since the cocycle $\Bar{\zeta}(\sigma)$ preserves the equation for $Q_{0}$ for every $\sigma \in \Gal(L/K)$, it follows that $Q$ has coefficients in $K$. 
\end{proof}

From \cref{twist as a conic}, $M$ gives a conic $C \subseteq \PP^2_{K}$ which is a twist of $C_0$ by $\zeta.$ The cocycle $\zeta$ is a coboundary if and only if $C$ has a rational point. We use \texttt{Magma} \cite{MR1484478} to check this.

Suppose that the conic $C$ has a rational point. Choose a lift \[\Tilde{\zeta} \colon \Gal(L/K) \to \GL_2(L).\] Using $\Tilde{\zeta}$ we define $\mu \colon \Gal(L/K) \times \Gal(L/K) \to L^{\times}$ as $$\mu(\sigma,\tau):=\Tilde{\zeta}(\sigma)(\sigma(\Tilde{\zeta}(\tau))(\Tilde{\zeta}(\sigma\tau)^{-1}).$$ It can be easily verified that $\mu$ is a $2$-cocycle. 

\begin{proposition}
The cocycle $\zeta \colon \Gal(L/K) \to \PGL_2(L)$ is a coboundary if and only if the $2$-cocycle $\mu \colon \Gal(L/K) \times \Gal(L/K) \to L^{\times}$ is a $2$-coboundary. 
\end{proposition}

\begin{proof}
We will show that if $\zeta$ is a coboundary, then $\mu$ is a $2$-coboundary. Assume that $\zeta$ is a coboundary, i.e., there exists an element $A \in \PGL_2(L)$ such that $\zeta(\sigma)=A^{-1}\sigma(A)$ for every $\sigma \in \Gal(L/K).$ Assume that $\Tilde{\zeta}(\sigma)=\zeta(\sigma)\alpha(\sigma)$, where $\alpha(\sigma)$ is some element of $L^{\times}.$ Then we can check that $\mu(\sigma,\tau)=\alpha(\sigma)(\sigma(\alpha(\tau)))\alpha(\sigma\tau)^{-1}.$ Therefore, $\mu$ is a $2$-coboundary.

Now assume that $\mu$ is $2$-coboundary, i.e., there exists a function $\alpha \colon \Gal(L/K) \to L^{\times}$ such that $\mu(\sigma,\tau)=\alpha(\sigma)(\sigma(\alpha(\tau)))\alpha(\sigma\tau)^{-1}.$ Define $\phi \colon \Gal(L/K) \to \GL_2(L)$ as $\phi(\sigma)=\Tilde{\zeta}(\sigma)\alpha(\sigma)^{-1}.$ Using the definition of $\mu$ we can check that $\phi$ is a cocycle. Hilbert $90$ states that there exists a matrix $A \in \GL_2(L)$ such that $\phi(\sigma)=A^{-1}\sigma(A)$ for every $\sigma \in \Gal(L/K).$ So, it follows that $\zeta \colon \Gal(L/K) \to \PGL_2(L)$ is a coboundary.
\end{proof}

We know that $\mu$ is a coboundary if and only if there exists a map $$f \colon \Gal(L/K) \to L^{\times}$$ such that $$\mu(\sigma,\tau)=f(\sigma)\sigma(f(\tau))f(\sigma\tau)^{-1}.$$

We are interested in computing the function $f$ and the following Corollary $2.2$ from \cite{Preu2013} is useful.

\begin{proposition}\label{prop:f}
Suppose $\Gal(L/K)$ is a finite cyclic group of order $n$ generated by $\sigma$ and $\mu \colon \Gal(L/K) \times \Gal(L/K) \to L^{\times}$ is a $2$-cocycle. If $\mu$ is a $2$-coboundary then the associated function $f \colon \Gal(L/K) \to L^{\times}$ is given by the following formula: 
\[f(\sigma^{i}):=(\mu(\sigma^{i},\sigma^n))(\prod_{j=0}^{i-1}((\mu(\sigma^{j},\sigma)(\sigma^{j}.a')))^{-1}\] for every $i \in \{1,\ldots,n\}.$ Here $a'$ is an element of $L^{\times}$ such that norm of $a'$ is \[(\prod_{j=0}^{n-1}\mu(\sigma^{j},\sigma))^{-1}.\]
\end{proposition}

After we compute $f$ we define a function $\phi \colon \Gal(L/K) \to \GL_2(L)$ as $\phi(\sigma)=\Tilde{\zeta}(\sigma)f(\sigma)^{-1}.$ We can check that  $\phi$ is a 1-cocyle. Hilbert90 states that there exists a matrix $A \in \GL_2(L)$ such that $\phi(\sigma)=A^{-1}\sigma(A)$ for every $\sigma \in \Gal(L/K).$








\subsection{Computing Hilbert 90 matrices}\label{Computing H90 mat}

Let $K$ be a number field and let $L$ be a finite Galois extension of $K$ and define $G:= \Gal(L/K)$. Hilbert 90 states that $H^{1}(G,\GL_{n}(L))=1.$ Equivalently, given a cocycle $\psi \colon G \to \GL_{n}(L)$ there exists a matrix $A\in \GL_{n}(L)$ such that $\psi(\sigma)=A^{-1}\sigma(A)$ for every $\sigma \in G.$ We now describe how to find $A$.

\begin{itemize}
    \item Define the $L$-vector space $V:=L^n.$ Define a new $G$-action on $V$ by $\sigma*v=\psi(\sigma)\sigma(v)$; it acts $K$-linearly. Let $W$ be the $K$-subspace of $V$ fixed by this $G$-action. Using Proposition $7 (A.V.63)$ of \cite{MR1994218}, the natural map $L \tensor_{K} W \to V$ of $L$-vector spaces is an isomorphism. In particular, $W$ has dimension $n$ over $K$. 
    
    \item Fix a basis $\frakA$ for $L$ considered as a vector space over $K$ and a basis $\frakB$ for $V$ considered as a vector space over $L$. The set $B=\{b.v~|~b \in \frakA, v \in \frakB\}$ is a basis of $V$ over $K.$
    
    \item The $K$-linear map $V\to W$ that sends $v$ to $1/|G|(\Sigma_{g \in G} g*v)$ is a projection. Therefore, the set $C =\{1/|G|(\Sigma_{g \in G} g*u) : u \in B\}$ spans $W$ over $K$.
    
    \item Choose $n$ linearly independent vectors from $C$; they form a basis for $W$ and hence also a basis of $V$ over $L.$ Let $A \in \GL_n(L)$ be the change of basis matrix from our basis in C to the standard basis of $L^n$. We have $\psi(\sigma)=A^{-1}\sigma(A)$ for all $\sigma \in G.$  
\end{itemize}

\section{Fields of Definition}

We are interested in computing genus $0$ modular curves $(X_G,\pi_G)$ of prime power levels with a point such that $\det(G)$ is a proper subgroup of $\Zhat^{\times}.$ As explained in \cref{sec:Modular Curves} such a curve $(X_G,\pi_G)$ will be defined over a subfield of $K_{p^m}$, where $p^m$ is the level of $G.$ In this section we explain that it is sufficient to consider the subfields of $K_{p^n}$, where $p^n$ is the level of $G \intersect \SL_2(\Zhat).$ 
Let us discuss the procedure that we use to compute $(X_G,\pi_G).$ We fix a genus $0$ congruence subgroup $\Gamma$ of prime power level $p^n$ that contains $-I.$ Let $h$ be the normalized hauptmodul of $\Gamma.$ Please refer to section $4$ of \cite{2105.14623} for a discussion on hauptmoduls and their computations for any genus $0$ congruence subgroup. Since the coefficients of $q$-expansion of $h$ lie in $K_{p^n}$, the function field of $X_{\Gamma}$ is $K_{p^n}(h).$ There is a unique $\pi_{\Gamma}$ in $K_{p^n}(t)$ such that $\pi_{\Gamma}(h)=j$ where $j$ is the modular $j$-invariant. To compute $\pi_{\Gamma}$ we use the method described in section $4.4$ of \cite{MR3671434}. The function $\pi_{\Gamma}$ describes a morphism $\pi_{\Gamma} \colon X_{\Gamma} \to \mathbb{P}^1_{K_{p^n}}.$ 

Suppose there is an open subgroup $G$ of $\GL_2(\Zhat)$ containing $-I$ for which $G \intersect \SL_2(\ZZ)$ is $\Gamma.$ Let $p^m$ be the level of $G.$ We know that $m$ is greater than or equal to $n.$ There is an isomorphism $f \colon (X_{\Gamma})_{K_{p^m}} \to (X_{G})_{K_{p^m}}$ satisfying $\pi_{G} \circ f =\pi_{\Gamma}.$ We observe that $K^{\det(G)}_{p^m}(h)$ is a function field for a model of $X_{\Gamma}$, isomorphic to $\mathbb{P}^{1}_{K^{\det(G)}_{p^m}}.$ Since we have models for $X_{\Gamma}$ and $X_G$ over $K^{\det(G)}_{p^m}$ there is a natural action of $\Gal(K_{p^m}/K^{\det(G)}_{p^m})$ on $f.$ Using that $\pi_G$ is defined over $K^{\det(G)}_{p^m}$ we see that $f$ satisfies the following condition
\begin{equation}\label{equ:coc}
    \sigma(\pi_{\Gamma})=\pi_{\Gamma} \circ f^{-1}\circ \sigma(f)
\end{equation}
for every $\sigma \in \Gal(K_{p^m}/K^{\det(G)}_{p^m}).$

\cref{equ:coc} gives us a condition on the image of the cocycle \[\zeta \colon \Gal(K_{p^m}/K^{\det(G)}_{p^m}) \to \Aut((X_{\Gamma})_{K_{p^m}})=\PGL_2(K_{p^m})\] that sends $\sigma$ to $f^{-1}\circ \sigma(f).$ Further, we have also checked that image of $\zeta$ lies in $\PGL_2(K_{p^n})$, where $p^n$ is the level of $\Gamma.$ 

For a fixed $m$, there are only finitely many cocycles from $\Gal(K_{p^m}/K^{\det(G)}_{p^m})$ to $\PGL_2(K_{p^m})$ that satisfy \cref{equ:coc}. We compute the cocycles on a set of generators of $\Gal(K_{p^m}/K^{\det(G)}_{p^m})$ and extend to the whole group using the cocycle property. Let $\zeta$ be one such cocycle. From \cref{lemma:coboundary P^1} and discussion in the paragraph following it, we know that if $\zeta$ is a coboundary in $H^{1}(\Gal(K_{p^m}/K^{\det(G)}_{p^m}),\PGL_2(K_{p^m}))$, then $\zeta$ gives a pair $(\PP^1_{K^{\det(G)}_{p^m}},\pi)$ which is isomorphic to modular curve $(X_G,\pi_G.)$ 

We will now proceed to show that $X_G$ is defined over a subfield of $K_{p^n}$. Recall that $p^n$ is the level of congruence subgroup $G \intersect \SL_2(\ZZ).$

Suppose that $p$ is an odd prime. We know that $K_{p^M}$ is a cyclic extension of $\Q$ for any positive integer $M.$ Therefore, the number of subfields of $K_{p^M}$ that contain $\Q$ is equal to $\omega(p^{M-1}(p-1))$ which is equal to $M\omega(p-1).$ When we increase our field to $K_{p^{(n+1)}}$ from $K_{p^n}$ we get exactly $\omega(p-1)$ more subfields that correspond to divisors of $p-1$, i.e., corresponding to a divisor $d$ of $p-1$, we get a subfield $\Q \subseteq K \subseteq K_{p^{(n+1)}}$ such that $\Gal(K_{p^{(n+1)}}/K)$ is of order $d.$ In case the level of $\Gamma$ is $2^m$ for some positive integer $m$, then we assume that $m$ is bigger than or equal to $3$, if $m$ is equal to $1$ or $2$ then we can increase the field from $K_2$ or $K_4$ to $K_8.$ When we increase the field from $K_{2^m}$ to $K_{2^{m+1}}$ we get $3$ additional subfields of $K_{2^{m+1}}$ that contain $\Q.$ For each of these subfields $K$, $\Gal(K_{2^{m+1}}/K)$ is of order $2.$

Define $K':=K_{p^n} \intersect K.$ The Galois group $\Gal(K_{p^n}/K')$ is isomorphic to $\Gal(K_{p^{n+1}}/K).$ We know that $\pi_{\Gamma}$ is defined over $K_{p^n}$ so cocycles from $\Gal(K_{p^{n+1}}/K) \to \PGL_2(K_{p^{n+1}})$ satisfying \cref{equ:coc} are same as cocycles from $\Gal(K_{p^n}/K') \to \PGL_2(K_{p^n})$ satisfying \cref{equ:coc}. Therefore, we can assume that such a genus $0$ modular curve $(X_G,\pi_G)$ such that $G \intersect \SL_2(\ZZ)$ is defined over a subfield of $K_{p^n}.$ 

\subsection{Extending the upper field}

Fix a genus $0$ congruence subgroup $(\Gamma,\pi_{\Gamma})$ containing $-I$ of level $p^n.$ If the level is $2$ or $4$, we extend the level to $8.$ Fix a field $\Q \subseteq K \subseteq K_{p^n}.$ Let $K_{p^{\infty}}$ denote the compositum of fields $K_{p^m}$ with $m$ greater than or equal to $1.$ In this section, we will explain that there are finitely many cocycles from $\Gal(K_{p^{\infty}}/K) \to \PGL_2(K_{p^{\infty}})$ that satisfy \cref{equ:coc}. Let $\Aut_{K_{p^{\infty}}}(X_{\Gamma},\pi_{\Gamma})$ denote the group of isomorphisms $f \colon (X_{\Gamma})_{K_{p^{\infty}}} \to (X_{\Gamma})_{K_{p^{\infty}}}$ that satisfy $\pi_{\Gamma} \circ f= \pi_{\Gamma}.$ We have checked that for every genus $0$ congruence subgroup $\Gamma$ containing $-I$ the group $\Aut_{K_{p^{\infty}}}(X_{\Gamma},\pi_{\Gamma})$ is defined over $K_{p^n}$, i.e., $\Aut_{K_{p^{\infty}}}(X_{\Gamma},\pi_{\Gamma})=\Aut_{K_{p^n}}(X_{\Gamma},\pi_{\Gamma}).$ 

\begin{lemma}\label{vertical cocycles}

Let $b$ be the least common multiple of orders of elements in $\Aut_{K_{p^n}}(X_{\Gamma},\pi_{\Gamma})$ that divide some power of $p.$ Any cocycle from $\Gal(K_{p^{\infty}}/K) \to \PGL_2(K_{p^{\infty}})$ that satisfy \cref{equ:coc} factors through $\Gal(K_{bp^n}/K).$ 

\end{lemma}

\begin{proof}

Let $\zeta \colon \Gal(K_{p^{\infty}}/K) \to \PGL_2(K_{p^{\infty}})$ be a cocycle that satisfies \cref{equ:coc}. We observe that the condition given by \cref{equ:coc} means that the image of $\zeta$ on the subgroup $\Gal(K_{p^{\infty}}/K_{p^n})$ lies in $\Aut_{K_{p^n}}(X_{\Gamma},\pi_{\Gamma})$ because $\pi_{\Gamma}$ is defined over $K_{p^n}.$ Choose a large enough positive integer $M$ such that $\zeta$ factors through $\Gal(K_{p^M}/K_{p^n}).$ Recall that if $p=2$, then we are assuming that $n$ is bigger than or equal to $3.$ Therefore, the Galois group $\Gal(K_{p^M}/K_{p^n})$ is cyclic and its cardinality has only $p$ as a prime factor. Let $g$ be a generator of $\Gal(K_{p^M}/K_{p^n}).$ Let $a:=\zeta(g).$ Then by our choice of $b$ we get that $a^b=1.$ Also observe that $\zeta$ is a homomorphism on $\Gal(K_{p^M}/K_{p^n})$ so, $\zeta(g^b)=1.$ It follows that $\zeta$ factors through $\Gal(K_{bp^n}/K).$  
\end{proof}

\section{Proof of Main theorem}

In this section, we give a proof of \cref{Mainthm}.

We fix a genus $0$ congruence subgroup $\Gamma$ of level $p^n$ containing $-I.$ Let $h$ be the normalized hauptmodul of $\Gamma$ and let $\pi_{\Gamma} \in K_{p^n}(t)$ be the rational function such that $\pi_{\Gamma}(h)=j.$ We also fix a field $\Q \subseteq K \subseteq K_{p^n}.$ Let $b_{\Gamma}$ be the integer $b$ as in \cref{vertical cocycles} corresponding to $\Gamma.$ 
We then compute all the cocycles from $\Gal(K_{b_{\Gamma}p^n}/K) \to \PGL_2(K_{b_{\Gamma}p^n})$ that satisfy \cref{equ:coc}. There are finitely many such cocycles. Let $\zeta$ be one such cocycle. We then check if $\zeta$ is a coboundary and compute a matrix $A \in \PGL_2(_{b_{\Gamma}p^n})$ using method explained in \cref{sec:Twists}. The pair $(\PP^1_{K},\pi_{\Gamma} \circ A^{-1})$ gives us the modular curve corresponding to the cocycle $\zeta.$

\section{Reading the tables}

In this section, we give a brief guide to go through the tables.
In \href{https://github.com/Rakvi6893/Genus-0-modular-curves-over-higher-number-fields/blob/main/Tables%20for%20odd%20prime%20power%20level_genus%200%20over%20extensions%20of%20Q.txt}{ text file 1} 
we list the modular curves $(X_G,\pi_G)$ where level of $G$ is a power of an odd prime number. The label is of the form $NA-MB$ where $A$ and $B$ are alphabets, $NA$ is the Cummins Pauli \cite{MR2016709} label of congruence subgroup corresponding to $G \intersect \SL_2(\Zhat)$ and $M$ is the level of $G.$ We list the function $\pi_{\Gamma} \in K_N(t)$ satisfying $\pi_{\Gamma}(h)=j$ where $h$ is the hauptmodul of congruence subgroup corresponding to $G \intersect \SL_2(\Zhat).$ We indicate the field $K$ over which $(X_G,\pi_G)$ is defined and give a matrix $A$ such that $\pi_G=\pi_{\Gamma}(A(t)).$

In \href{https://github.com/Rakvi6893/Genus-0-modular-curves-over-higher-number-fields/blob/main/Table_powers%20of%202.txt}{ text file 2}
we list the modular curves $X_G$ where level of $G$ is a power of $2$. We describe $X_G$ as a conic $Q$. To obtain $\pi_G$ explicitly we can do the following. We can compute a change of variable matrix $B$ between $Q$ and $C_0$. Recall that $C_0$ is the conic $y^2-xz=0.$ This feature is not currently implemented in \texttt{Magma}. We can then compute $C:=\bar{\phi}^{-1}(BA)$ and $\pi_G$ will be $\pi_{\Gamma}(C^{-1}(t)).$

\end{document}